\documentclass[11pt, reqno]{amsart}
\usepackage{lipsum}
\usepackage{a4wide}
\usepackage{amssymb} 
\usepackage{amsmath}
\usepackage{amsthm} 
\usepackage{tikz} 
\usepackage{amscd}
\usepackage{hyperref}

\usepackage{stmaryrd}
\usepackage{enumitem}
\hypersetup{colorlinks,linkcolor={blue},citecolor={blue},urlcolor={red}}
\theoremstyle{plain}
\usepackage{comment}
\numberwithin{equation}{section}

\newcommand{\rank}{\operatorname{rank}}

\newtheorem{theorem}{Theorem}[section]
\newtheorem{corollary}[theorem]{Corollary} 
\newtheorem{lemma}[theorem]{Lemma}
\newtheorem{remark}[theorem]{Remark}
\newtheorem{proposition}[theorem]{Proposition}

\newtheorem{definition}[theorem]{Definition}

\newcommand{\gitquot}{\mathbin{\backslash\!\!\backslash}}

\title{On the GIT quotient of Grassmannians by one dimensional torus}

\author[B.~N.~Chary]{Narasimha Chary Bonala}
\address{Narasimha Chary Bonala\\
Department of Mathematics and Statistics, Indian Institute of Technology Kanpur,
U.P. India, 208016.
}
\email{chary@iitk.ac.in}

\author[S.~S.~Kannan]{S.~Senthamarai Kannan}
\address{S. Senthamarai Kannan\\
Chennai Mathematical Institute, Plot H1, SIPCOT IT Park, Siruseri, Kelambakkam,
603103, India.
}
\email{kannan@cmi.ac.in}

\author[S.~Pattanayak]{Santosha Pattanayak}
\address{Santosha Pattanayak\\
Department of Mathematics and Statistics, Indian Institute of Technology Kanpur,
U.P. India, 208016.
}
\email{santosha@iitk.ac.in}

\begin{document}


\begin{abstract} We consider the action of the one-parameter subgroup of the special linear group corresponding to a simple root on Grassmannians and describe the structure of the associated Geometric Invariant Theory (GIT) quotients with respect to Pl\"ucker line bundle. Using the combinatorics of Weyl group elements, we explicitly describe the semistable loci and identify cases where the resulting quotient admits the structure of a parabolic induction of a projective space. We further analyze the orbit structure under the Levi subgroup, compute the Picard group, connected component of the automorphism group and examine key geometric features such as Fano properties, cohomology of line bundles, and projective normality with respect to the descended linearization. 
\end{abstract}

\maketitle

\section{Introduction}

The action of a maximal torus $T \subset SL(n, \mathbb C)$ on the Grassmannian $\mathrm{G}(r,n)$, and more generally on the flag variety $G/P$, has been a central object of study in algebraic geometry, with deep connections to representation theory, symplectic geometry, and moduli theory. The resulting GIT quotients, referred to as \emph{weight varieties} by Allen Knutson in his thesis \cite{Knu}, capture rich geometric and combinatorial structure. In \cite{HK}, Hausmann and Knutson established that the GIT quotient $\mathrm{G}(2,n)\mathbin{/\mkern-5mu/} T$ of the Grassmannian $G(2,n)$ by $T$ can be interpreted as the moduli space of polygons in 
$\mathbb{R}^3$. They further showed that this quotient can also be realized as 
the GIT quotient of an $n$-fold product of projective lines by the diagonal 
action of $\mathrm{PSL}(2,\mathbb{C})$. More generally, using the 
Gel'fand--MacPherson correspondence, the GIT quotient of $G(r,n)$ by $T$ may be 
identified with the GIT quotient of $(\mathbb{P}^{r-1})^n$ by the diagonal 
action of $\mathrm{PSL}(r,\mathbb{C})$, thus providing a moduli-theoretic perspective rooted in classical invariant theory.

Kapranov's work (\cite{Kap1}, \cite{Kap2}) further established that the Grothendieck--Knudsen moduli space $\overline{M}_{0,n}$ arises as the Chow quotient of the Grassmannian $\mathrm{G}(2,n)$ by the torus action, connecting these constructions to the geometry of stable rational curves. In parallel, foundational results by Dabrowski \cite{D}, Carrell-Kurth \cite{CK}, and Howard \cite{How} demonstrated that closures of $T$-orbits in the flag variety $G/P$ are normal and projectively normal, ensuring well-behaved embeddings and coordinate rings. Together, these developments provide a robust framework for analyzing torus quotients of flag varieties, unifying themes from GIT, moduli theory, and representation theory, and paving the way for new geometric and combinatorial insights.

In \cite{GW}, Giansiracusa and Wu studied the diagonal subtorus orbits in the Grassmannian. They established that the closures of orbits of a diagonal subtorus acting on the Grassmannian $G(r,n)$ correspond bijectively to representable discrete polymatroids of rank $r$ on multisets determined by the subtorus structure. They construct the associated Chow quotient, generalizing Kapranov’s approach, and show that this quotient compactifies a space of configurations of linear subspaces with prescribed coincidences. They extended the classical Gelfand–MacPherson correspondence and Gale duality to this new setting, and identified cases where the Chow quotient is birational to moduli spaces such as the Chen–Gibney–Krashen spaces. 

However, the Geometric Invariant Theory quotients of Grassmannians (and more generally of flag varieties) with respect to subtori of a maximal torus have received comparatively little attention in the literature. Recently, in \cite{ORSW}, the authors investigate the relationship between the Chow and the GIT quotients of a projective variety by the $\mathbb C^*$-action and construct explicitly the Chow quotient from the GIT quotient through successive blowups under some assumptions. While the GIT quotients by a maximal torus capture rich combinatorial structures related to weight polytopes and matroid theory, the study of subtorus actions introduces new geometric phenomena. These quotients interpolate between the well understood quotients by a maximal torus and the more rigid setting of Levi subgroup quotients, offering a refined perspective on stability conditions, orbit closures, and degenerations.

\subsection{GIT quotients by one-parameter subgroups:}

We now focus on the most fundamental case of a subtorus action that of a one-parameter subgroup.  

Let $G$ be a connected semisimple algebraic group over $\mathbb{C}$, with a fixed maximal torus $T$ of $G$ and a Borel subgroup $B$ of $G$ containing $T$ and corresponding root system $R \subset X^*(T)$.  
Let $\lambda : \mathbb{G}_m \to T$ be a one-parameter subgroup of $T$, that is, an element of the cocharacter lattice $Y(T) = \operatorname{Hom}(\mathbb{G}_m, T)$.  
For simplicity of notation, we denote $\lambda(\mathbb G_m)$ by $\lambda$. Such a one-parameter subgroup defines an algebraic $\mathbb{G}_m$-action on any projective $G$-variety equipped with a $T$-linearization. In particular, for any parabolic subgroup $P \subset G$, the induced $\mathbb{G}_m$-action on the partial flag variety $G/P$ provides a natural framework for Geometric Invariant Theory (GIT) once an ample line bundle $\mathcal{L}(\chi)$ associated to a dominant character $\chi \in X^*(T)$ is fixed.

Since $C_G(T) = T$ and $N_G(T)/T$ is finite, there is no natural action of a 
positive-dimensional connected reductive group arising from a subgroup of $G$ on the invariant coordinate ring
\[
\bigoplus_{d \in \mathbb{Z}_{\geq 0}} H^0(G/P, \mathcal{L}(d\chi))^T .
\]
In contrast, for a one-parameter subgroup $\lambda$, the reductive group 
$C_G(\lambda)$ acts naturally on 
\[
\bigoplus_{d \in \mathbb{Z}_{\geq 0}} H^0(G/P, \mathcal{L}(d\chi))^{\lambda}.
\]
This setting admits a richer structure, as the presence of a reductive group action allows us to employ both representation–theoretic and geometric methods. In 
particular, one may consider the induced action of $C_G(\lambda)$ on the GIT quotient 
\[
\lambda \backslash\!\backslash (G/P)^{ss}_{\lambda}(\mathcal{L}(\chi)) ,
\]
where $(G/P)^{ss}_{\lambda}(\mathcal{L}(\chi))$ denotes the $\lambda$-semistable 
locus with respect to $\mathcal{L}(\chi)$.

Studying this quotient as a $C_G(\lambda)$-variety is an interesting problem in its 
own right, as it connects invariant theory with the geometry of partial flag varieties. 

This perspective has been developed in several recent works. In \cite{K1}, \cite{K2}, \cite{KP09} and \cite{KS09}, the action of one-parameter subgroups $\lambda$ has been used to study the GIT quotients of Schubert varieties and flag varieties by a maximal torus.  
More recently, Ghosh and Kannan \cite{GK} analyzed the GIT quotients of Schubert varieties $X(w) \subset G/P$ under the action of the one-dimensional torus $\lambda_s(\mathbb G_m)$ associated with the simple root $\alpha_s$, where $P$ is a maximal parabolic subgroup of $G$. They provide a combinatorial criterion for the existence of semistable points with respect to a $\lambda_s(\mathbb G_m)$-linearized line bundle $\mathcal{L}(\chi)$, and show that there exists a unique minimal Schubert variety for which the semistable locus is non-empty. Specializing to $G = PSL(n, \mathbb{C})$ and line bundles associated with multiples of minuscule fundamental weights, they give an explicit description of the GIT quotient of this minimal Schubert variety as a projective space of complex matrices. Furthermore, for the action of $\lambda_s$ on the Grassmannian $G(r,n)$, they showed that $G(r,n)_{\lambda_s}^{ss}(\mathcal L(n\varpi_r)) =G(r,n)_{\lambda_s}^s(\mathcal L(n\varpi_r))$ if and only if $n$ does not divide $rs$.

In this paper, we consider the action of $\lambda:=n\lambda_s$ on $G(r,n)$ linearized by the ample line bundle $\mathcal{L}(\varpi_r)$ corresponding to the fundamental weight $\varpi_r$, where $\lambda_s$ is the coweight correspond to the fundamental weight $\varpi_s$. We analyze the structure of the quotient
$\lambda \gitquot G(r,n)_{\lambda}^{ss}(\mathcal{L}(\varpi_r))$ as a $H$-variety where $H$ is the semisimple part of the centralizer of $\lambda$ in $SL(n, \mathbb C)$. Note that $\lambda=n\lambda_s$. 

The geometry of these quotients depends crucially on the relative position of the parameters $r$, $s$, and $n$.  
Let $p = \lfloor \frac{rs}{n} \rfloor$.  
From \cite{GK}, it is known that the Schubert variety $X(w_{s,r})$ associated to the Weyl group element
\[
w_{s,r} = (s_s \cdots s_{p+1})(s_{s+1}\cdots s_{p+2})\cdots(s_{s+r-p-1}\cdots s_r)
\]
plays a key role in the description of the semistable locus of the $\lambda(\mathbb G_m)$-linearized line bundle $\mathcal L(\varpi_r)$ on $G(r,n)$. It is the minimal Schubert variety in $G(r, n)$ for which the semistable locus is non-empty. 
The combinatorics of the reduced expression of 
\[
w_0^{S \setminus \{\alpha_r\}} = \tilde{w}.w_{s,r}\quad \text{such that}\quad l(\tilde{w})+l(w_{s,r})=l(w_0^{S \setminus \{\alpha_r\}}),
\]
where $w_0^{S \setminus \{\alpha_r\}}$ is the minimal representative of the coset $w_0W_{S \setminus \{\alpha_r\}}$, controls how the semistable locus extends from $X(w_{s,r})$ to the entire Grassmannian.

A particularly transparent situation arises when $p = 0$ or $p=r + s - n$, provided $r+s \geq n$.  
In these cases, the semistable locus of $G(r,n)$ is the $H = SL(s,\mathbb C) \times SL(n-s, \mathbb C)$–saturation of the semistable locus of the Schubert variety $X(w_{s,r})$ (see Proposition \ref{prop:semilocus}).  
In \cite{GK} it is proved that $\lambda \gitquot X(w_{s,r})_{\lambda}^{ss}(\mathcal{L}(\varpi_r))$ is the projective space $\mathbb{P}(M_{s-p, r-p})$ of the $s-p \times r-p$ matrices. We give an explicit description of the full quotient (Theorem \ref{thm:structure}):
\[
\lambda \gitquot G(r,n)_{\lambda}^{ss}(\mathcal{L}(\varpi_r)) \;\simeq\; H \times_{P_{s,r}} \mathbb{P}(M_{s-p, r-p}),
\]
where $P_{s,r}$ denotes the stabilizer of $X(w_{s,r})$ in $H$.  
Thus, in these cases, the quotient denoted by $X$ is realized as a \emph{parabolic induction} of a projective space, or equivalently, a homogeneous fiber bundle over $H/P_{s,r}$.
This structure enables us to study its geometric properties. Specifically, we describe the $H$-orbit decomposition of the quotient variety $X$ (Proposition \ref{prop:Starti}), compute its Picard group (Corollary \ref{cor:Pic}), and we prove that $X$ is Fano (Theorem \ref{thm:fano}). 
Moreover, we compute the connected component of its automorphism group (Theorem \ref{thm:aut}) containing the identity automorphism using Brion’s theory of relative automorphisms of homogeneous bundles (see \cite{Br}).

We also analyze the cohomology of line bundles on $X$ (Theorem \ref{thm:cohomology}), determine the structure of the global sections as an $H$-module (Theorem \ref{thm:G decomposition}), and establish the projective normality of $X$ with respect to the descent of the line bundle $\mathcal{L}(\varpi_r)$ from the Grassmannian (Theorem \ref{thm:projective_normality}).  
These results collectively provide a detailed geometric and representation-theoretic understanding of the quotient $\lambda \gitquot G(r,n)_{\lambda}^{ss}(\mathcal{L}(\varpi_r))$.

The paper is organized as follows. 
In \textbf{Section 2}, we recall the basic preliminaries on the action of one-parameter subgroups and review the necessary tools from Geometric Invariant Theory.
In \textbf{Section 3}, we analyze the structure of the GIT quotient for the action of $\lambda$. Using the combinatorics of the Weyl group element $w_{s,r}$, we describe the semistable loci explicitly and identify cases in which the quotient admits the structure of a parabolic induction of a projective space. We also provide a detailed description of the orbits of the Levi subgroup $H$ and their closures. 
In \textbf{Section 4}, we study the geometric properties of the quotient. In particular, we compute its Picard group, show that the quotient variety is Fano, and analyze the cohomology of  line bundles. We further describe the connected component of the automorphism group, the $H$-module decomposition of global sections of line bundles, and prove projective normality with respect to the descended Plücker linearization.


\section{Preliminaries}

 In this section, we set up some preliminaries and notation. We refer to \cite{Hum}, \cite{Hum1}, \cite{Spr} for preliminaries in Lie algebras and algebraic groups. Let $G$ be a semi-simple algebraic group over $\mathbb C$. We fix a maximal torus $T$ of $G$ and a Borel subgroup $B$ of $G$ containing $T$.  Let $N_G(T)$ be the normaliser of $T$ in $G$. Let $W = N_G(T)/T$ be the Weyl group of $G$ with respect to $T$. Let $R$ denote the set of roots with respect to $T$ and $R^+$ be the set of positive roots with respect to $(T,B)$. Let $S = \{\alpha_1,\alpha_2,\ldots,\alpha_n\} \subset R$ be the set of simple roots and for a subset $I\subseteq S$ we denote by $P_I$  the parabolic subgroup of $G$ generated by $B$ and $\{n_{\alpha}: \alpha \in I\}$, where $n_{\alpha}$ is a representative of $s_{\alpha}$ in $N_{G}(T)$. Let $X(T)$ (resp. $Y(T)$) denote the set of characters of $T$ (resp. one-parameter subgroups of $T$). Let 
 \[E_1 := X(T) \otimes \mathbb{R}, \quad  E_2 := Y(T) \otimes \mathbb{R}.\] Let $\langle.,.\rangle: E_1 \times E_2 \to \mathbb{R}$ be the canonical non-degenerate bilinear form. Let \[\overline{C(B)}:=\{\nu\in E_{2}|\ \langle \alpha,\nu\rangle\geq0,\ \text{for all} ,\ \alpha\in R^{+}\}.\] For all $\alpha \in R$, there is an injective homomorphism $\phi_\alpha : SL_2 \to G$ such that
	  	   \[
           \check{\alpha}(t)=\phi_{\alpha} ( \left[ {\begin{array}{cc}
   t & 0\\
   0 & t^{-1} \\
  \end{array} } \right]).\] We also have $s_{\alpha}(\chi) = \chi - \langle\chi,\check{\alpha}\rangle \alpha$ for all $\alpha \in R$ and $\chi \in E_1$. Set $s_i = s_{\alpha_i}$ for all $i = 1, 2, \ldots,n$. Let $\{\varpi_i: i=1, 2, \ldots, n\} \subseteq E_1$ be the fundamental weights; i.e.   
           $\langle\varpi_i,\check{\alpha_j}\rangle = \delta_{ij}$ for all $i,j = 1,2,\ldots n$ and $\{\lambda_i: i=1,2,\cdots,n\} \subseteq E_2$ be the set of fundamental coweights, i.e.   
           $\langle\alpha_i,\lambda_j \rangle = \delta_{ij}$ for all $i,j = 1,2,\ldots n$. Let $U$ be the unipotet radical of $B$. For each root $\alpha \in R$, we denote by $U_{\alpha}$, the $T$-stable root subgroup of $B$ corresponding to $\alpha$ for the conjugation action.

	There is a natural partial order $\le$ on $X(T)$ defined by $\psi\leq \chi$ if and only if $\chi-\psi$ is a nonnegative integral linear combination of simple roots.
	Let $u_{\alpha}:\mathbb{C}\longrightarrow U_{\alpha}$ be the isomorphism such that $tu_{\alpha}(a)t^{-1}=u_{\alpha}(\alpha(t)a)$, for all $t\in T$, $a\in \mathbb{C}$.

 For a simply connected semi-simple algebraic group $G$ and for a parabolic subgroup $P$ of $G$, the quotient space $G/P$ is a homogeneous space for the left action of $G$. The quotient $G/P$ is called a generalized flag variety. When $G=SL(n, \mathbb C)$ and $P_{S \setminus \{\alpha_r\}}$ is the maximal parabolic subgroup corresponding to the simple root $\alpha_r$, the quotient can be identified with $G(r,n)$, the Grassmannian of $r$ dimensional vector subspaces of $\mathbb C^n$. 

\subsection{Hilbert-Mumford Criterion} Let $G$ be a reductive group acting on a projective variety $X$. Let $\nu$ be a one-parameter subgroup of $G$. Let $\mathcal{L}$ be a $G$-linearized very ample line bundle on $X$.
	\begin{enumerate}
		\item 
		Let $x\in \mathbb{P}(H^{0}(X,\mathcal{L})^{*})$ and $\hat{x}$ be a point in the cone $\widehat{X}$ over X which lies on $x$.  Let $\{v_{i}: 1\leq i\leq k\}$ be a basis of $H^{0}(X,\mathcal{L})^{*}$ such that $\nu(t)\cdot v_{i}=t^{m_{i}}v_{i}$, for $1\leq i\leq k$. Write $\hat{x}=\displaystyle\sum_{i=1}^{k}c_{i}v_{i}$, $c_i \in \mathbb C$ for $1 \leq i \leq k$. Then the Hilbert-Mumford numerical function is defined by $$\mu^{\mathcal{L}}(x,\nu):= -\displaystyle \min_{i}\{m_{i}| c_{i}\neq 0\}.$$
		\label{definition of HM weight}
		\label{HM criterion}
		\item 
		 The set of semistable points is defined as \[ X^{ss}_{G}(\mathcal{L})=\{x\in X|\ \exists \  s \ \in H^{0}(X,\mathcal{L}^{\otimes n})^{G}\ \text{for some}\ n\in \mathbb N\ \text{such that}\ s(x)\neq 0\}.\]
			\label{semistability def}

			\item The set of stable points is defined as
			$X^{s}_{G}(\mathcal{L})=\{x\in X^{ss}_{G}(\mathcal{L})| $ the orbit $G\cdot x$ is closed in $X^{ss}_{G}(\mathcal{L})$ and the stabilizer $G_{x}$ of $x$ in $G$ is finite$\}.$
		
	\end{enumerate}
	We then have the Hilbert-Mumford criterion:
	\begin{theorem}$($see \cite[Theorem 2.1]{MFK}$)$\label{thm:Hilbert-Mumford} Let $x\in X$. Then 
		\begin{enumerate}
			\item x is semi-stable if and only if $\mu^{\mathcal{L}}(x,\nu)\geq 0$ for every one parameter subgroup $\nu$ of $G$.
			\item x is stable if and only if $\mu^{\mathcal{L}}(x,\nu)> 0$ for every non trivial one parameter subgroup $\nu$ of $G$. 
		\end{enumerate}
		\label{HM theprem for G}
	\end{theorem}
		Let $\chi=\displaystyle\sum_{i=1}^{n}m_{i}\omega_{i}$ be a non trivial dominant character of $T$. Let $J=\{\alpha_{i}\in S:m_{i}=0\}$. Let $P=P_{J}$.
	For $w\in W^{J}$, let $X(w):=\overline{BwP/P}$ denote the Schubert variety in $G/P$ corresponding to $w$.

	Now, we recall a Lemma  due to C. S. Seshadri. This will be used for computing semistable points.
	\begin{lemma} $($see \cite[Lemma 5.1]{CSS}$)$\label{SeS2}
		\begin{enumerate}
		\item Let $x=bwP/P$. 
		Let $\nu\in \overline{C(B)}$ be a one-parameter subgroup of $G$.  Then we have \[\mu^{\mathcal{L}(\chi)}(x,\nu)=-\langle w(\chi),\nu\rangle.\]
		 
\item Given a finite set $A$ of non trivial one parameter subgroups of $T$,  there is an ample line bundle $\mathcal{L}$ on $G/B$ such that $\mu^{\mathcal{L}}(x, \nu)\neq 0$ for all $x\in G/B$, and for all $\nu \in A$. \end{enumerate}

		\label{CSS}
	\end{lemma}
	
\begin{remark} \label{rem} In the above lemma, the sign is negative because we are using the left action of $P$ on $G/P$ while in \cite[Lemma 5.1]{CSS} the action is on the right. By imitating the
proof for the opposite Borel $B^-$, for $x \in B^-wP/P, w \in W^J$ and for every $\nu\in \overline{C(B)}$ we have \[\mu^{\mathcal{L}(\chi)}(x,-\nu)=\langle w(\chi),\nu\rangle.\]
    
\end{remark}


\section{Structure of the quotient}
In this section, we describe the structure of the GIT quotient and give a stratification of the quotient in terms of orbits.

\subsection{} Throughout we work with the special linear group $SL(n, \mathbb{C})$. Since $SL(n, \mathbb{C})$ is not of adjoint type, for any $1 \leq s \leq n-1$, we work with the one-parameter subgroup $\lambda: \mathbb G_m \rightarrow T \subset SL(n, \mathbb{C})$, with the following parameterization.
\[
\lambda(t) = \operatorname{diag}(\underbrace{t^{n-s}, \dots, t^{n-s}}_{s \text{ times}}, \underbrace{t^{-s}, \dots, t^{-s}}_{n-s \text{ times}}).
\]
Note that $\lambda=n\lambda_s$, where $\lambda_s$ is the coweight correspond to the fundamental weight $\varpi_s$. The centralizer $ C_{SL(n,\mathbb{C})}(\lambda)$ of $\lambda$ is given by
\[
\left\{ 
\begin{pmatrix}
A_1 & 0 \\
0 & A_2
\end{pmatrix} \in SL(n, \mathbb{C}) \,\middle|\,
A_1 \in GL(s, \mathbb{C}),\, A_2 \in GL(n-s, \mathbb{C}),\, \det(A_1)\det(A_2) = 1 
\right\}.
\]
In other words, the centralizer is isomorphic to the group
\[
S\big(GL(s, \mathbb{C}) \times GL(n-s, \mathbb{C})\big).
\]
The semi-simple part of the centralizer is $H :=SL(s, \mathbb{C}) \times SL(n-s, \mathbb{C})$.

Let $\mathcal L(\varpi_r)$ be the line bundle associated with the fundamental weight $\varpi_r$. Set
$$
X := \lambda \gitquot G(r,n)_{\lambda}^{ss}(\mathcal{L}(\varpi_r)).
$$
Then, we have a natural action of $H$ on $X$ induced from the standard action of $SL(n, \mathbb C)$ on $G(r,n)$. So we study the structure of $X$ as a $H$-variety.

\subsection{Sphericality of $X$} Let $G$ be an arbitrary connected reductive group.  Recall that a normal $G$-variety $Z$ is called \emph{spherical} if the action of a Borel subgroup $B$ on $Z$ has a dense open orbit.

We have the following result from \cite[Lemma 5.4]{RP14}.
\begin{lemma}\label{lem:double-flag-spherical}
Let $G$ be a connected reductive group.  
Let $P, Q \subseteq G$ be parabolic subgroups and let $K$ be a Levi subgroup of $P$.  
Then the following conditions are equivalent:
\begin{enumerate}
  \item $G/Q$ is a $K$-spherical variety,
  \item $G/P \times G/Q$ is a spherical $G$-variety with respect to the diagonal action of $G$.
\end{enumerate}
\end{lemma}

For any simple, simply connected group $G$, Littelmann classified the pairs of parabolic subgroups $(P, Q)$ for which $G/P \times G/Q$ is a spherical $G$-variety with respect to the diagonal action of $G$ (see \cite[p.~144, Table~I]{Peter}). For $G = SL(n, \mathbb{C})$, we  have the following result.

\begin{theorem}\label{thm:Littelmann}
 For any two maximal parabolic subgroups $P, Q$ of $G=SL(n, \mathbb C)$, the space $G/P \times G/Q$ is a spherical $G$-variety with respect to the diagonal action of $G$. 
\end{theorem}

\begin{proposition}\label{prop:sphericality}
The GIT quotient $X$ is a spherical $H$-variety.
\end{proposition}

\begin{proof}
Consider the maximal parabolic subgroups $P_{S\setminus \{\alpha_r\}}$ and $P_{S\setminus \{\alpha_s\}}$ of $G=SL(n, \mathbb C)$. By Theorem \ref{thm:Littelmann} together with Lemma \ref{lem:double-flag-spherical}, it follows that the Grassmannian $G(r, n)$ is spherical for the Levi subgroup $L_{P_{S\setminus \{\alpha_s\}}}= C_{G}(\lambda)$ of $P_{S\setminus \{\alpha_s\}}$. Therefore, we conclude that $X$ is a $H$ -spherical variety.  
\end{proof}

We have the following duality result which enable us to reduce the study of the quotient to $r \leq \lfloor\frac{n}{2}]$. 

\begin{proposition}\label{Prop:dual}
Let $\mathcal L(\varpi_r)$ (resp. $\mathcal L(\varpi_{n-r})$ be the line bundles on $G(r,n)$ (resp. $G(n-r,n)$) associated to the weights $\varpi_r$ (resp. $\varpi_{n-r}$). Let $\lambda'=n\lambda_{n-s}$. Then we have
$$\lambda \gitquot G(r,n)_{\lambda}^{ss}(\mathcal{L}(\varpi_r)) \simeq \lambda' \gitquot G(n-r,n)_{\lambda'}^{ss}(\mathcal{L}(\varpi_{n-r})).$$
\end{proposition}

\begin{proof}
Consider the natural duality isomorphism:
\[
\phi:G(r,n) \rightarrow G(n-r,n), \quad V \mapsto V^\perp,
\] where $V^\perp \subset \mathbb{C}^n$ denotes the orthogonal complement of $V$ with respect to the standard bilinear form.  Under this duality the line bundle \( \mathcal{L}(\varpi_r) \) corresponds naturally to \( \mathcal{L}(\varpi_{n-r}) \), $\varpi_r \mapsto -w_0(\varpi_r)=\varpi_{n-r}$ and $\lambda_s \mapsto -w_0(\lambda_s)=\lambda_{n-s}$, where $w_0$ is the longest element of the Weyl group $W$. Using Hilbert-Mumford numerical criterion, we get that \[
x \in G(r,n)^{ss}_{\lambda}(\mathcal{L}(\varpi_r)) \,\, \text{if and only if} \,\, w_0.\phi(x)\in G(n-r,n)^{ss}_{\lambda'}(\mathcal{L}(\varpi_{n-r})).
\]
Therefore, the above isomorphism restricts to an isomorphism of semistable loci, and since the duality isomorphism is $G$-equivariant and carries the $\lambda$–action on $G(r,n)$ to the $\lambda'$–action on $G(n-r,n)$, it descends to an isomorphism of the corresponding GIT quotients:
\[
\lambda \gitquot G(r,n)^{ss}_{\lambda}(\mathcal{L}(\varpi_r)) \cong \lambda' \gitquot G(n-r,n)^{ss}_{\lambda'}(\mathcal{L}(\varpi_{n-r})).
\]
This completes the proof. \end{proof}

 \subsection{The structure of $X$}

 We first recall that by \cite[Lemma 3.8]{GK}, there exists a unique minimal Schubert variety for which the semistable locus is non-empty, denoted by $X(w_{s,r})$, where \[w_{s,r}=(s_s \cdots s_{p+1})(s_{s+1}\cdots s_{p+2})\cdots (s_{s+r-p-1}\cdots s_r) \,\, \text{and} \,\, p=\lfloor \frac{rs}{n} \rfloor.\] Note that $0 \leq p \leq n-2$.

  The following lemma describes the semi-stable locus of the Grassmannian $G(r,n)$ with respect to the line bundle $\mathcal{L}(\varpi_r)$ under the $\lambda$-action. The description is via the semi-stable locus of $X(w_{s,r})$.
   Define \[
   A:=\Big\{(v, \varphi)\in W^{S\setminus\{\alpha_r\}}\times W^{S\setminus\{\alpha_r\}}~:~ v\ngeq w_{s,r}, \varphi\geq w_{s,r} ~\text{and}~ v \leq \varphi \Big\}. 
      \] 
    For $w\in W^{S\setminus \{\alpha_r\}}$ such that $w\geq w_{s,r}$, we define
      \[
      A_{w}:=\Big\{(v, \varphi)\in A ~:~ \varphi \leq w \Big\}.
      \]
      Then, 
\begin{lemma}\label{ss-general} Let $R_{v, \varphi}:=(U^{-}_vvP/P) \cap (U_{\varphi}\varphi P/P)$ be the Richardson cell. 
\begin{enumerate}
    \item We have \[G(r,n)_{\lambda}^{ss}(\mathcal{L}(\varpi_r))=\sqcup_{(v, \varphi)\in A}R_{v, \varphi}.\]
    \item For $w\in W^{S\setminus \{\alpha_r\}}$ with $w\geq w_{s,r}$, we have 
    \[X(w)_{\lambda}^{ss}(\mathcal{L}(\varpi_r))=\sqcup_{(v, \varphi)\in A_w}R_{v, \varphi}.\]
\end{enumerate}
  
\end{lemma}
\begin{proof}
   Fix $x\in R_{v, \varphi}$. By Lemma \ref{SeS2} and Remark \ref{rem}, for any $w\in W^{S\setminus \{\alpha_r\}}$ we have, 
   
      \begin{enumerate}
          \item[(i)] for any $y \in B^-wP/P$, we have \[\mu^{\mathcal{L}(\omega_r)}(y,-n\lambda_{s})=\langle w(\omega_{r}),n\lambda_{s}\rangle.\] 
       \item [(ii)] for any $z\in BwP/P$, we have 
       \[\mu^{\mathcal{L}(\omega_r)}(z, n\lambda_{s})= -\langle w(\omega_{r}), n\lambda_{s}\rangle.\] 
\end{enumerate}
Hence, by Theorem \ref{thm:Hilbert-Mumford} (Hilbert-Mumford criterion) and by (i) and (ii), we have 
\[x\in G(r,n)_{\lambda}^{ss}(\mathcal{L}(\varpi_r))\] if and only if 
   \[
   \langle v(\omega_{r}),n\lambda_{s}\rangle \geq 0\quad \text{and}\quad \langle \varphi (\omega_{r}),n\lambda_{s}\rangle \leq 0.\]
   Therefore, we get $$x\in G(r,n)_{\lambda}^{ss}(\mathcal{L}(\varpi_r))\quad\text{if and only if}\quad v\ngeq w_{s,r} ~\text{and}~\varphi\geq w_{s,r}.$$ Hence $R_{(v, \varphi)}$ is non-empty and is contained in $G(r,n)_{\lambda}^{ss}(\mathcal{L}(\varpi_r))$. This proves (1). 
   
   Since $X(w)_{\lambda}^{ss}(\mathcal{L}(\varpi_r))=X(w) \cap G(r,n)_{\lambda}^{ss}(\mathcal{L}(\varpi_r))$, the assertion (2) follows. 
\end{proof}

We then have the following corollary describing the quotient. 

\begin{corollary}
We have an isomorphism of GIT quotients
\[
\lambda \gitquot G(r,n)_{\lambda}^{ss}(\mathcal{L}(\varpi_r))
\simeq \bigsqcup_{(v, \varphi)\in A} \bigl( \lambda \gitquot R_{v, \varphi} \bigr).
\]
\end{corollary}
\begin{proof}
By Lemma \ref{ss-general}, the semistable locus of the $\lambda$–action on $G(r,n)$
decomposes as a disjoint union of Richardson cells:
\[
G(r,n)_{\lambda}^{ss}(\mathcal{L}(\varpi_r))
= \bigsqcup_{(v, \varphi)\in A} R_{v, \varphi}.
\]
Each $R_{v,\varphi}$ is stable under the action of $\lambda$, since both
$U^-_v vP/P$ and $U_{\varphi}\varphi P/P$ are $\lambda$–invariant subvarieties.
Hence, the quotient of the whole semistable locus
decomposes as the disjoint union of the quotients of its $\lambda$–stable pieces:
\[
\lambda \gitquot G(r,n)_{\lambda}^{ss}(\mathcal{L}(\varpi_r))
\simeq \bigsqcup_{(v, \varphi)\in A} (\lambda \gitquot R_{v, \varphi}).
\]
This establishes the claimed isomorphism.
\end{proof}

Let $\tilde{w} \in W$ be such that $w_0^{S \setminus \{\alpha_r\}}=\tilde ww_{s,r}$ with $l(w_0^{S \setminus \{\alpha_r\}})=l(\tilde w)+l(w_{s,r})$, where $w_0^{S \setminus \{\alpha_r\}}$ is the minimal representative of the longest element $w_0$ relative to the Weyl group $W_{S \setminus \{\alpha_r\}}$ of the maximal parabolic subgroup $P_{S \setminus \{\alpha_r\}}$.

\begin{lemma}\label{lem:extention}
    Let $\tilde{w}$ be as above. Then $s_s\nleq \tilde w$ if and only if either $p=0$ or $p=r+s-n$ with $r+s-n\geq 0$.
\end{lemma}
\begin{proof}
Recall that 
\[
w_{s,r} = (s_s \cdots s_{p+1})(s_{s+1} \cdots s_{p+2}) \cdots (s_{s+r-p-2} \cdots s_{r-1})(s_{s+r-p-1} \cdots s_r),
\]
and
\[
w_0^{S \setminus \{\alpha_r\}} = (s_{n-r} \cdots s_1)(s_{n-r+1} \cdots s_2) \cdots (s_{n-2} \cdots s_{r-1})(s_{n-1} \cdots s_r).
\]

We can write 
\[
w_0^{S \setminus \{\alpha_r\}} = \tilde{w} \cdot w_{s,r},
\]
with 
\[
\tilde{w} =
\begin{cases}
(s_{n-r} \cdots s_{s+1})(s_{n-r+1} \cdots s_{s+2}) \cdots (s_{n-2} \cdots s_{r+s-1})(s_{n-1} \cdots s_{r+s}), & \text{if } p=0, \\
(s_{n-r} \cdots s_1) \cdots (s_{n-r-1+p} \cdots s_p)(s_{n-r+p} \cdots s_{s+1}) \cdots (s_{n-1} \cdots s_{s+r-p}), & \text{if } p \geq 1.
\end{cases}
\]

In the case $p=0$, it is clear that $s_s$ does not appear in $\tilde{w}$. 

For $p \ge 1$, the simple reflection $s_s$ appears in $\tilde{w}$ if and only if $n - r + p > s$, hence $s_s$ does not appear if and only if 
\[
n - r + p \le s.
\]

On the other hand, by definition, we have 
\[\left\lfloor \frac{rs}{n} \right\rfloor - (r+s-n) \ge r+s-n,
\]
since 
\[
\frac{rs}{n} - (r+s-n) = \frac{(n-r)(n-s)}{n} \ge 0.
\]

Combining the inequalities, we conclude that for $p \ge 1$, the simple reflection $s_s$ does not appear in $\tilde{w}$ if and only if
\[
n - r + p = s.
\]
\end{proof}

\subsection{} When $p = 0$ or $p=r + s - n$ with $r+s-n\geq 0$, then the semistable locus of the $\lambda$-action on $G(r,n)$ admits an explicit geometric description. 
In these cases, we show that the GIT quotient is a parabolic induction of a projective space.

\begin{proposition}\label{prop:semilocus} Assume that $p=0$ or $p=r+s-n$ with $r+s-n\geq 0$.
Let $w=w_{s,r}$ and $P=P_{S \setminus \{\alpha_r\}}$. We have
$$
G(r,n)_{\lambda}^{ss}(\mathcal{L}(\varpi_r)) = H \cdot (U_w wP/P \setminus \{wP/P\}).
$$
\end{proposition}

\begin{proof} 
We first claim that $X(w)_{\lambda}^{ss}(\mathcal{L}(\varpi_r)) = U_w wP/P \setminus \{wP/P\}$. By Lemma 2.2, for any $x = b w P/P \in U_w wP/P$  we have
\[
\mu^{\mathcal L(\varpi_r)}(x, \lambda) = -\langle w(\varpi_r), \lambda \rangle.
\]
Since $w$ is the minimal element for which the semistable locus is non-empty, we have 
\[
\langle w(\varpi_r), \lambda \rangle \leq 0,
\qquad
\text{and} \qquad
\langle v(\varpi_r), \lambda \rangle > 0 
\text{ for all } v < w.
\]
Hence,
\[
\begin{cases}
\mu^{\mathcal L(\varpi_r)}(x, \lambda) < 0, & \text{if } v < w, \\
\mu^{\mathcal L(\varpi_r)}(x, \lambda) \geq 0, & \text{if } v = w.
\end{cases}
\]

At the torus-fixed point \(x_0 = wP/P\), the fiber of \(\mathcal{L}(\varpi_r)\) has 
\(T\)-weight \(w(\varpi_r)\), so along \(\lambda\) the induced weight is 
\(-\langle w(\varpi_r), \lambda\rangle \geq 0\). Further, $\mu^{\mathcal L(\varpi_r)}(x, -\lambda) = \langle w(\varpi_r), \lambda \rangle <0$ for every $x \in B^-wP/P$. Since $x_0=wP/P \in BwP/P \cap B^-wP/P$, we have $\mu^{\mathcal L(\varpi_r)}(x_0, -\lambda) <0$. So $x_0$ is not semistable. 

Also, since $\mu^{\mathcal L(\varpi_r)}(x, \lambda)=-\langle v(\varpi_r), \lambda \rangle <0$ for every $x \in BvP/P$ and for all $v <w$ in $W^{S \setminus \{\alpha_r\}}$, every element of $BvP/P$ is $\lambda$-unstable. Hence, we have \[
X(w)^{ss}_{\lambda}(\mathcal{L}(\varpi_r)) 
\subseteq U_w wP/P \setminus \{wP/P\}.
\]
 Now let $x \in U_w wP/P \setminus \{wP/P\}$. Then there exists $v \in W^{S \setminus \{\alpha_r\}}$ with $v <w$ such that $x \in B^-vP/P$. Therefore, $\mu^{\mathcal L(\varpi_r)}(x, -\lambda)=\langle v(\varpi_r), \lambda \rangle \geq 0$. We already have $\mu^{\mathcal L(\varpi_r)}(x, \lambda)=-\langle w(\varpi_r), \lambda \rangle \geq 0$ and so $x$ is semistable.
Thus, we get
\[
X(w)^{ss}_{\lambda}(\mathcal{L}(\varpi_r)) 
= U_w wP/P \setminus \{wP/P\}.
\]
Now let $x \in G(r, n)_{\lambda}^{ss}(\mathcal{L}(\varpi_r))$. Then there exists an element $\phi \in W^{S \setminus \{\alpha_r\}}$ such that
$$ x \in U_{\phi} \phi P / P.$$
Note that $\phi \geq w$. Thus, we can write $\phi = v w$ for some $v \in W$ such that $vw \in W^{S \setminus \{\alpha_r\}}$ with $l(\phi) = l(v) + l(w).$ We claim that $v \in W_{S \setminus \{\alpha_s\}}$. 

Since $p=0$ or $p=r+s-n$, by Lemma~\ref{lem:extention}, we extend $w_{s,r}$ to $w_0^{S \setminus \{\alpha_r\}}$ as $
w_0^{S \setminus \{\alpha_r\}} = \tilde{w} \cdot w_{s,r},$ with $s_s\nleq \tilde w$.
Thus, the simple reflection $s_s$ does not appear in any reduced expression of $\tilde{w}$, the claim follows. 

Therefore, for any $\beta \in R^+(\tilde{w}^{-1})$ we have $\alpha_s \nleq \beta$. Thus, there is an element $g \in SL(s) \times SL(n-s)$ and an element $y \in \bigl(U_w w P / P \setminus \{w P / P\}\bigr)$ such that $x=gy$. Hence, $$x \in H.\bigl(U_w w P / P \setminus \{w P / P\}\bigr).$$
Conversely, let $x \in H \cdot (U_w wP/P \setminus \{wP/P\})$. Since every element of $(U_w wP/P \setminus \{wP/P\})$ is semi-stable and $G(r,n)_{\lambda}^{ss}(\mathcal{L}(\varpi_s))$ is $H$-stable, we obtain $x \in G(r,n)_{\lambda}^{ss}(\mathcal{L}(\varpi_s))$. This proves the proposition. 
\end{proof}


\subsection{Parabolic Induction} 
We first recall the construction of a parabolic induction.

Let $G'$ be a connected reductive algebraic group, and let $P' \subseteq G'$ 
be a parabolic subgroup with Levi decomposition $P' = L'U'$, 
where $L'$ is a Levi subgroup and $U'$ is the unipotent radical of $P'$.  If $Z$ is an $P'$-variety, the {\it parabolic induction} of $Z$ to $G'$ 
is the $G'$-variety defined by
\[
G' \times_{P'} Z :=(G' \times Z)/\sim,
\]
where the equivalence relation is given by
\[
(g,z) \sim (gp'^{-1}, ~ p'z), \quad \text{for all } p' \in P', ~z\in Z~.
\]

Note that this construction equips $G' \times_{P'} Z$  with a natural 
left $G'$-action by multiplication on the first factor.

Assume that either $p=0$ or $p=r+s-n$ when $r+s \geq n$. Now let $P_{s,r} = \text{Stab}_G(X(w_{s,r})) \cap H,$ and let $\mathbb{P}(M_{s-p,r-p})$ be the projectivization of the space of $s-p \times r-p$ matrices. Then, 

\begin{theorem}\label{thm:structure}
We have the following isomorphism:
$$
X \cong H \times_{P_{s,r}} \mathbb{P}(M_{s-p,r-p}).
$$
\end{theorem}

\begin{proof}
Since $X(w_{s,r})$ is the minimal Schubert variety admitting semistable points for $\lambda$, by \cite[Theorem 1.3]{GK}, we have
$$
\lambda \gitquot X(w_{s,r})^{\mathrm{ss}}_{\lambda}(\mathcal{L}(\varpi_r)) = \mathbb{P}(M_{s-p,r-p}).
$$
Define
\[
\Phi : H \times_{P_{s,r}} \mathbb{P}(M_{s-p,r-p}) \longrightarrow X,
\qquad [(h,y)] \longmapsto h \cdot y,
\]
where $y \in \mathbb{P}(M_{s-p,r-p})$ is viewed as the $\lambda$–orbit of some 
$x \in X(w_{s,r})^{\mathrm{ss}}_{\lambda}(\mathcal{L}(\varpi_r))$.

We claim that $\Phi$ is an isomorphism.

\noindent\underline{Injectivity of $\Phi$:} 
Suppose $\Phi([(h_1, y_1)]) = \Phi([(h_2, y_2)])$. 
Then $h_1 \cdot y_1 = h_2 \cdot y_2$ in $X$. 
Choose $x_i \in X(w_{s,r})^{\mathrm{ss}}_{\lambda}(\mathcal{L}(\varpi_r))$ such that $\pi(x_i) = y_i$, 
where $$\pi : X(w_{s,r})^{\mathrm{ss}}_{\lambda}(\mathcal{L}(\varpi_r)) \to \mathbb{P}(M_{s-p,r-p})$$ is the quotient map.
The equality $h_1 \cdot y_1 = h_2 \cdot y_2$ means that the $\lambda$–orbits 
of $h_1 x_1$ and $h_2 x_2$ coincide, 
so there exists $t \in \mathbb{G}_m$ such that
\[
h_1 x_1 = \lambda(t) \cdot h_2 x_2.
\]
Hence,
\[
(h_2^{-1} h_1) \cdot x_1 = \lambda(t) \cdot x_2.
\]
Since $\pi(\lambda(t) \cdot x_2) = \pi(x_2) = y_2$, 
we may replace $x_2$ by $\lambda(t) \cdot x_2$ (it represents the same 
class $y_2$). Thus we may assume that
\[
(h_2^{-1} h_1) \cdot x_1 = x_2.
\]
As $x_1, x_2 \in X(w)$, so $h_2^{-1} h_1$ stabilizes $X(w)$. 
Since $h_2^{-1} h_1 \in H$, we obtain 
\[
h_2^{-1} h_1 \in H \cap \operatorname{Stab}_G(X(w_{s,r})) = P_{s,r}.
\]
Hence $h_1 \equiv h_2 \pmod{P_{s,r}}$, proving injectivity.

\medskip
\noindent\underline{Surjectivity of $\Phi$:} 
By Proposition~\ref{prop:semilocus}, 
every semistable point of $X$ is $H$–equivalent to a point in 
$X(w_{s,r})^{\mathrm{ss}}_{\lambda}(\mathcal{L}(\varpi_r))$. 
Thus, every point of $X$ arises as $h \cdot y$ for some 
$h \in H$ and $y \in \mathbb{P}(M_{s-p,r-p})$, proving surjectivity.

Moreover, since the base field is $\mathbb C$ and $X$ is normal, by Zariski's main theorem, the map $\Phi$ is an isomorphism (see \cite[$\text{III}_2$, Theorem ~ 4.4.3, p. 136]{EGAIII}). This completes the proof.
\end{proof}

\subsection{\bf The $H$-orbits on $X$}

In this subsection, we describe the orbits of the action of $H$ on $X$. For $1 \leq t \leq \min\{r-p,s-p\}$, let us denote the quotient by
$$
Y_t := \lambda \gitquot \left( \{ A \in M_{s-p,r-p} : \text{rank}(A) = t \} \setminus \{0\} \right).
$$
\begin{proposition}\label{prop:Starti} We have:
\begin{enumerate}
\item For $1 \leq t \leq \min\{r-p,s-p\}$, the $H$-orbit in $X$ is given by $O_t$, where
$$
O_t = H \times_{P_{s,r}} Y_t.
$$
\item The number of $H$-orbits in $X$ is $\min\{r-p,s-p\}$.
\item We have $O_{t_1} \subset \overline{O}_{t_2}$ if and only if $t_1 \leq t_2$.
\end{enumerate}
\end{proposition}

\begin{proof}
By Theorem~\ref{thm:structure}, we have
$$
X \cong H \times_{P_{s,r}} \mathbb{P}(M_{s-p,r-p}).
$$
Thus, the number of $H$-orbits in $X$ is equal to the number of $P_{s,r}$-orbits in $\mathbb{P}(M_{s-p,r-p})$. Let $U_{s,r}$ be the unipotent radical of $P_{s,r}$. By Borel fixed point theorem, we have $$\mathbb{P}(M_{s-p,r-p})^{U_{s,r}} \neq \emptyset.$$ Since the character group of $U_{s,r}$ is trivial, it follows that $M_{s-p,r-p}^{U_{s,r}} \neq \{0\}$. Note that $M_{s-p,r-p}$ is an irreducible $P_{s,r}$-module as the $P_{s,r}$-action on $M_{s-p,r-p}$ factors through $SL(s-p,\mathbb{C}) \times SL(r-p,\mathbb{C})$ and $M_{s-p,r-p}$ is an irreducible $SL(s-p,\mathbb{C}) \times SL(r-p,\mathbb{C})$-module. Hence the action of $U_{s,r}$ on $M_{s-p,r-p}$ is trivial.

Further, by Schur's lemma, the center of the Levi subgroup of $P_{s,r}$ acts as a nonzero scalar on $M_{s-p,r-p}$. Therefore, the action of the radical of $P_{s,r}$ on $\mathbb{P}(M_{s-p,r-p})$ is trivial. Consequently, the number of $P_{s,r}$-orbits in $\mathbb{P}(M_{s-p,r-p})$ is equal to the number of $SL(s-p,\mathbb{C}) \times SL(r-p,\mathbb{C})$-orbits in $\mathbb{P}(M_{s-p,r-p})$. By standard linear algebra arguments, we see that the number of $SL(s-p,\mathbb{C}) \times SL(r-p,\mathbb{C})$-orbits in $\mathbb{P}(M_{s-p,r-p})$ is $\min\{r-p, s-p\}$ and these orbits are classified by the rank of the matrices with the required relation. Thus, we conclude the proof.
\end{proof}

We deduce the following:
\begin{corollary}
For $1 \leq t \leq \min\{r-p,s-p\}$, define $X_t := H \times_{P_{s,r}} Z_t$, where $$Z_t=\lambda \gitquot \left( \{ A \in M_{s-p,r-p} : \text{rank}(A) \leq t \} \setminus \{0\} \right).$$ 
\begin{enumerate}
\item We then have $X_t=\overline{O_t}$.
    \item We have the stratification of $X$:
    $$
    X_1 \subset X_2 \subset \cdots \subset X_{\min\{r-p,s-p\}-1} \subset X_{\min\{r-p,s-p\}} = X.
    $$
    \item The orbit closure $X_1$ is a closed orbit, and
    $$
    X \setminus O_{\min\{r-p,s-p\}} = X_{\min\{r-p,s-p\}-1}.
    $$
\end{enumerate}
\end{corollary}

\begin{proof}
    Recall that
$Y_t=\{[A]\in\mathbb P(M_{s-p,r-p}) : \rank(A)=t\}$ and 
$O_t=H\times_{P_{s,r}}Y_t$.

(1) Since $\overline{Y_t}=\{[A]\in\mathbb P(M_{s-p,r-p}):\rank(A)\le t\}$, 
we have
\[
\overline{O_t}
=\overline{H\cdot ( \, \{1\}\times Y_t \,) }
=H\cdot\overline{\{1\}\times Y_t}
=H\times_{P_{s,r}}\overline{Y_t}.
\]
The last equality follows from the homogeneous-bundle description 
$$H\times_{P_{s,r}}\mathbb P(M_{s-p,r-p})\to H/P_{s,r}.$$
By definition $Z_t$ is the $\lambda$–quotient of 
$\overline{Y_t}\setminus\{0\}$, hence $H\times_{P_{s,r}}Z_t$ is the image of 
$H\times_{P_{s,r}}\overline{Y_t}$ in the GIT quotient. Therefore
\[
X_t \;=\; H\times_{P_{s,r}}Z_t \;=\; H\times_{P_{s,r}}\overline{Y_t}
\;=\; \overline{O_t},
\]
as required.

(2) The inclusions $\overline{Y_{t}}\subset \overline{Y_{t+1}}$ for all $t$ 
is obvious. Applying the 
functor $H\times_{P_{s,r}}(-)$ (which preserves inclusions) yields
\[
X_1 \subset X_2 \subset \cdots \subset X_{\min\{r-p,s-p\}}.
\]
By Proposition \ref{prop:Starti} $X_{\min\{r-p,s-p\}}$ 
equals the whole $X$, so the displayed chain is a stratification of $X$.

(3) The variety $X_1$ corresponds to the unique 
closed $SL(s-p, \mathbb C)\times SL(r-p, \mathbb C)$–orbit in $\mathbb P(M_{s-p,r-p})$ consisting of rank $1$ matrices. Hence, 
$X_1$ is a closed $H$–orbit in $X$.  
Finally the open dense orbit is $O_{\min\{r-p,s-p\}}$, so its complement 
is the union of all lower strata; in particular, the complement of the top 
orbit equals $X_{\min\{r-p,s-p\}-1}$, which is the union of closures of orbits 
of rank at most $\min\{r-p,s-p\}-1$. This gives the proof of (3).
\end{proof}

\begin{proposition}
For $1 \leq t \leq \min\{r-p,s-p\}$, the variety $X_t$ has the following properties:
\begin{enumerate}
    \item $X_t$ is normal and Cohen–Macaulay.
    \item $X_t$ is a $H$-spherical variety.
\end{enumerate}
\end{proposition}

\begin{proof}

Since $Z_t$ is a determinantal variety, it is normal and Cohen--Macaulay (see \cite{LR}), and hence $X_t$ is also normal and Cohen--Macaulay.

Next, we show that $X_t$ is $H$-spherical for all $1 \leq t \leq \min\{r-p,s-p\}$. By Proposition \ref{prop:sphericality}, $X$ is $H$-spherical. Since the $X_t$'s are closed $H$-stable subvarieties of $X$, it follows that each $X_t$ is an $H$-spherical variety (see \cite[Theorem 3.1.19]{Perrin}).

We can also see the sphericality in a more elementary way, as we explain below: We first show that ${Z_t}$ is $SL(s-p, \mathbb C) \times SL(r-p, \mathbb C)$-spherical. 
Recall that the semisimple part of $P_{s,r}$ is  $SL(s-p, \mathbb C) \times SL(r-p, \mathbb C)$.
 Since $X_t$ is the parabolic induction $H \times_{P_{s,r}} Z_{t}$, it follows that $X_t$ is $H$-spherical. 
The group $SL(s-p,\mathbb{C}) \times SL(r-p,\mathbb{C})$ acts
on $M_{s-p,r-p}$ and hence on $Z_t$ by
\[
(C,D)\cdot A = CAD^{-1}, \,\, C \in SL(s-p,\mathbb{C}), \,\, D \in SL(r-p,\mathbb{C}), \,\,  A \in M_{s-p,r-p}.
\]
Let $B_{s-p}^-$ be the Borel of lower-triangular matrices in $\mathrm{SL}(s-p, \mathbb C)$ and $B_{r-p}$ be the Borel of upper-triangular matrices in $\mathrm{SL}(r-p, \mathbb C)$.
Let $B':=B_{s-p}^- \times B_{r-p}$. We show that the above action of $SL(s-p,\mathbb{C}) \times SL(r-p,\mathbb{C})$
on $\overline{Y_t}$ has a dense $B'$-orbit. Let 
\[{J}_t:=\begin{pmatrix} I_t & 0 \\[4pt] 0 & 0\end{pmatrix}\in M_{s-p,r-p},
\] where $I_t$ is the $t \times t$ identity matrix. We show that the orbit $B'.J_t$ is dense in $\overline{Y_t}$. 

To that end, we compute the dimension of the stabilizer $Stab_{B'}(J_t)$. Note that \[(C,D)\cdot J_t =J_t\quad \text{implies}\quad CJ_t=J_tD\quad \text{for} \quad C \in B_{s-p}^-, ~ D \in B_{r-p}.\] 
We write $C$ and $D$ in block decomposition form compatible with the \((t,s-p-t)\) row and \((r-p,r-p-t)\) column partitions:
\[
C=\begin{pmatrix} C_{11} & 0 \\ C_{21} & C_{22}\end{pmatrix}\in B_{s-p}^-,\qquad
D=\begin{pmatrix} D_{11} & D_{12}\\ 0 & D_{22}\end{pmatrix}\in B_{r-p},
\]
The condition $CJ_t=J_tD$ translates to 
\[
C_{11}=D_{11},\qquad C_{21}=0,\qquad D_{12}=0.
\]
Since $C_{11}$ is a lower triangular matrix whereas $D_{11}$ is a upper-triangular matrix, from $C_{11}=D_{11}$ we get that both are diagonal matrices. Again, since $C_{22}$ and $D_{22}$ are lower and upper triangular matrices of sizes $(s-p-t) \times (s-p-t)$ and $(r-p-t) \times (r-p-t)$, respectively, we obtain \[
\dim (Stab_{B'}(J_t)) \;=\; t \;+\; \frac{(r-p-t)(r-p-t+1)}{2} \;+\; \frac{(s-p-t)(s-p-t+1)}{2}-2.
\]
On the other hand, we have \[
dim (B')=\frac{(r-p)(r-p+1)}{2}+\frac{(s-p)(s-p+1)}{2}-2.
\]
Thus,
\begin{align*}
\dim(B' \cdot J_t) 
&= dim (B')- dim(Stab_{B'}(J_t)) \\
&= t(r + s - 2p - t) \\
&= dim(\{ A \in M_{s-p,r-p} : \text{rank}(A) \leq t \}).
\end{align*}
Note that $C_{w_{s,r}}=M_{s-p,r-p}$ and $\lambda \gitquot M_{s-p, r-p}=\mathbb P(M_{s-p, r-p})$. 
Consider the quotient map $$\pi: M_{s-p,r-p}\setminus \{0\} \rightarrow \mathbb P(M_{s-p,r-p}).$$ Therefore, the orbit $B'.\pi(J_t)$ is an open and dense subvariety in $Z_t$. 
Hence, $Z_{t}$ is $SL(r-p, \mathbb C) \times SL(s-p, \mathbb C)$-spherical. This completes the proof. 
\end{proof}

Now we recall the notion of wonderful varieties, due to Luna \cite{Luna96}.
\begin{definition}[Luna] Let $G$ be a connected reductive group.
A smooth, projective algebraic $G$-variety $X$ is called a {\it wonderful $G$-variety} if it satisfies the following properties:
\begin{enumerate}
    \item $X$ has an open dense $G$-orbit (denoted by $X_{G}^{0}$), i.e., $X=\overline{G\cdot x}$ for some $x\in X$.
    \item The complement $X\setminus X_{G}^{0}$ is a smooth normal crossing divisor (i.e., $$X\setminus X_{G}^{0}=X_{1}\cup X_{2}\cup\cdots\cup X_{r},$$ where $X_{j}\subseteq X$ is irreducible, smooth, $G$-divisor such that $X_{1},\ldots, X_{r}$ intersect transversally). We call $X_{i}$'s the boundary divisors.
    \item For every $I\subseteq\{1,\ldots,r\}$, $$\displaystyle\bigcap_{i\in I}X_{i}=\overline{G\cdot x_{I}}$$ for some $x_{I}\in X$. Moreover,  all the closures of $G$-orbits in $X$ are of this form.
\end{enumerate}
\end{definition}

\begin{definition}
The rank of a wonderful $G$-variety $X$ is defined to be the number of boundary divisors of $X$.
\end{definition}

Akhiezer \cite{Akh} and Brion \cite{Br1}, independently, established that wonderful varieties of rank $1$ can be characterized by the existence of a completion by homogeneous divisors, see \cite[Proposition 30.4]{Tim}. In our case, we obtain the following.
\begin{corollary}
Let $n\geq 5$ and $r = s = 2$, the GIT quotient $X$ is a wonderful $H$-variety with exactly two orbits. In particular, the rank of $X$ is one.
\end{corollary}

\begin{proof}
For $r = s = 2$, the quotient $X$ has exactly two orbits by Proposition~\ref{prop:Starti}. Moreover, one orbit is open, and its complement is a closed orbit that is also a divisor. Therefore, the quotient $X$ is a wonderful $H$-variety, see  \cite[Proposition 30.4]{Tim}.
\end{proof}

For $r=s=2$ and $n=3$ or $4$, we have $p=1$. In these two cases, we describe the quotient explicitly. 

\begin{proposition} Let $r=s=2$.
\begin{enumerate}
    \item For $n=3$, we have $X \cong \mathbb P^1$.
    \item For $n=4$, we have $X \cong \mathbb P^3$.
\end{enumerate}
   \end{proposition}
\begin{proof} (1) For $n=3$ and $r=s=2$, we have $p=1$. Here we set $\lambda=3\lambda_2$ and so we have 
 \[
\lambda(t)=\operatorname{diag}(t,t,t^{-2}).
\]
We now explicitly describe the quotient $$
X :=  \lambda\gitquot G(2,3)_{\lambda}^{ss}(\mathcal{L}(3\varpi_2)).
$$ The Grassmannian $G(2,3)$ has dimension $2$ and coincides with $\mathbb P^2$ under the Plücker embedding.  
Its homogeneous coordinates are
\[
[p_{12}:p_{13}:p_{23}].
\]
The weights of $\lambda_2$ are
\[
\operatorname{wt}(p_{12})=2,\qquad 
\operatorname{wt}(p_{13})=\operatorname{wt}(p_{23})=-1.
\]
The homogeneous coordinate ring of $G(2,3)$ is $
R = \mathbb C[p_{12},p_{13},p_{23}]$
and the $\lambda$-invariant monomials $p_{12}^a p_{13}^b p_{23}^c$ satisfy
\[
2a - b - c = 0, \qquad a,b,c \in \mathbb Z_{\geq 0}.
\]
Thus, the invariant ring is 
\[
R^{\lambda(\mathbb G_m)} \;=\; \mathbb{C}[p_{12}p_{13}^2,\, p_{12}p_{13}p_{23},\, p_{12}p_{23}^2]
\]
subject to the quadratic relation,
\[
(p_{12}p_{13}p_{23})^2 = (p_{12}p_{13}^2)(p_{12}p_{23}^2).
\]
This is the homogeneous coordinate ring of $\mathbb{P}^1$ under the quadratic Veronese embedding
\[
[x:y] \mapsto [x^2:xy:y^2]\;\in\; \mathbb{P}^2.
\]
Therefore,
\[
(X, \mathcal M) \;\cong\; (\mathbb{P}^1, \mathcal O(2)), 
\] where $\mathcal M$ is the descent of $\mathcal L(3\varpi_2)$ to $X$.

(2) For $n=4$ and $r=s=2$, we have $p=1$. Here we set $\lambda=4\lambda_2$ and so we have 
 \[
\lambda(t)=\operatorname{diag}(t^2,t^2,t^{-2},t^{-2}).
\] We describe the quotient $$
X :=  \lambda\gitquot G(2,4)_{\lambda}^{ss}(\mathcal{L}(2\varpi_2)).
$$
The Pl\"ucker coordinates $p_{ij}$ on $G(2,4)\subset \mathbb P^5$ transform with weights
\[
\operatorname{wt}(p_{12})=4,\qquad \operatorname{wt}(p_{34})=-4,\qquad
\operatorname{wt}(p_{13})=\operatorname{wt}(p_{14})=\operatorname{wt}(p_{23})=\operatorname{wt}(p_{24})=0.
\]
We set
\[
x_1:=p_{13},\quad x_2:=p_{14},\quad x_3:=p_{23},\quad x_4:=p_{24}.
\]
The Pl\"ucker relation 
$
p_{12}p_{34} - p_{13}p_{24} + p_{14}p_{23} \;=\; 0
$
transforms to
\[
p_{12}p_{34} \;=\; x_1x_4 - x_2x_3 \in \mathbb C[x_1,x_2,x_3,x_4].
\]
Let $R$ be the homogeneous coordinate ring of $G(2,4)$.  
Then the homogeneous coordinate ring of $X$ is 
\[
R^{\lambda(\mathbb G_m)} = \mathbb C[x_1,x_2,x_3,x_4].
\]
We then have,
\[
(X, \mathcal M) \;=\; \operatorname{Proj}\big(R^{\lambda(\mathbb G_m)}\big)
\;=\; \operatorname{Proj}\big(\mathbb C[x_1,x_2,x_3,x_4]\big)
\;=\; (\mathbb P^3, \mathcal O(1)), 
\] where $\mathcal M$ is the descent of $\mathcal L(2\varpi_2)$.  
\end{proof}

   \section{Geometric Properties of $X$}

   In this section, we study the geometric properties of $X$. Throughout this section we assume that either $p=0$ or $p=r+s-n$ with $r+s-n\geq 0$. So we have $X \cong H \times_{P_{s,r}} Y$, where $Y := \mathbb{P}(M_{s-p,\,r-p})$, thanks to Theorem \ref{thm:structure}. We set
\[
k:=
\begin{cases}
r+s, & \text{if } r+s\le n-1,\\[4pt]
r+s-n, & \text{if } r+s\ge n+1.
\end{cases}
\]
When $r+s=n$ we regard this as the boundary case. We first prove a general result.

\subsection{Picard group of $X$ and Fanoness} In this subsection, we compute the Picard group of $X$ and prove that $X$ is Fano. Recall that \[
w_{s,r}=(s_s \cdots s_{p+1})(s_{s+1}\cdots s_{p+2})\cdots(s_{s+r-p-1}\cdots s_r),
\]
where $p=\lfloor \frac{rs}{n}\rfloor$ and $P_{s,r}:=\operatorname{Stab}_G\bigl(X(w_{s,r})\bigr)\cap H$, where $H\cong SL(s,\mathbb C)\times SL(n-s, \mathbb C)$.

\begin{lemma}\label{lem:parabolic}
    The subgroup $P_{s,r}$ is a maximal parabolic subgroup of $H$ if $r+s \neq n$. When $r+s=n$, we have $P_{s,r}=H$.
\end{lemma}

\begin{proof}
    Recall that $P_{s,r} = \text{Stab}_G(X(w_{s,r})) \cap H$. 
    For any $1 \leq i \leq n-1$, we have 
    \[s_i.X(w_{s,r})=X(w_{s,r})\quad\text{if and only if}\quad w_{s,r}^{-1}s_iw_{s,r} \in W_P,\] where $P=P_{S \setminus \{\alpha_r\}}$.

If $r+s\leq n-1$, then the unique index $i$ for which
        $w_{s,r}^{-1}s_i w_{s,r}\notin W_P$ is $i=r+s$. Hence,
        $\operatorname{Stab}_G(X(w_{s,r}))=P_J$ with
        $J=S\setminus\{\alpha_{r+s}\}$.

If $r+s\geq n+1$, the same combinatorial computation (taking indices
        modulo $n$ and then reinterpreting them in $\{1,\dots,n-1\}$) shows
        that the unique exceptional index is $k=r+s-n$, so
        $\operatorname{Stab}_G(X(w_{s,r}))=P_J$ with
        $J=S\setminus\{\alpha_{r+s-n}\}$.
    
    If $r \neq n-s$, we have $\text{Stab}_G(X(w_{s,r}))=P_J$ where $J=S \setminus \{\alpha_{k}\}$. Therefore, $$P_{s,r} = \text{Stab}_G(X(w_{s,r})) \cap H$$ is a maximal parabolic subgroup of $H$. In fact $P_{s,r}=SL(s, \mathbb C) \times P'$, where $P'$ is the maximal parabolic subgroup of $SL(n-s, \mathbb C)$ corresponding to the simple root $\alpha_{k}$. 

On the other hand, when $r=n-s$ we have $w_{s,r}=w_0^{S \setminus \{\alpha_r\}}$. Therefore, the stabilizer of $X(w_{s,r})=G(r,n)$ is $SL(n,\mathbb C)$. Hence, we have $P_{s,r}=H$.
\end{proof}

\begin{lemma}\label{lem:linebundle}
Let $\mathcal{L}$ be any line bundle on $X$.
\begin{enumerate}
    \item If $r \neq n - s$, then there exist integers $a, b \in \mathbb{Z}$ such that
    \[
        \mathcal{L} \;\simeq\; \mathcal{L}(b\varpi_k) \otimes \mathcal{O}_Y(a).
    \]
    \item If $r = n - s$, then
    \[
        \mathcal{L} \;\simeq\; \mathcal{O}_Y(a)
        \qquad \text{for a unique } a \in \mathbb{Z}.
    \]
\end{enumerate}
\end{lemma}

\begin{proof}
By Theorem~\ref{thm:structure}, we have $X \cong H \times_{P_{s,r}} Y$. Thus, $X$ is a (Zariski) locally trivial fibration with fiber $Y$.
Since both $H/P_{s,r}$ and $Y$ are smooth projective rational varieties, and the Picard group $Pic(Y)$ is a projective $\mathbb{Z}$-module, it follows from \cite[Proposition 2]{BD20} that there is an isomorphism
$$
Pic(X) \cong Pic(H/P_{s,r}) \oplus Pic(Y).
$$
Now, since $H = SL(s, \mathbb C) \times SL(n - s, \mathbb C)$ and $Y = \mathbb{P}(M_{s-p, r-p})$, the result follows from Lemma \ref{lem:parabolic}.

When $r=n-s$, $P_{s,r}=H$, $X\cong H\times_H Y\cong Y$,
and every line bundle on $X$ is the pullback of a unique power of the tautological
bundle on $Y$:
\[
\mathcal L \;\simeq\; \mathcal O_Y(a)\quad\text{for a unique} \quad a\in\mathbb Z.
\]
\end{proof}

\begin{corollary}\label{cor:Pic}
    The Picard group of $X$ is given by
\[
\operatorname{Pic}(X) \;\cong\;
\begin{cases}
\mathbb{Z} \oplus \mathbb Z, & \text{if } r \neq n-s, \\[6pt]
\mathbb{Z}, & \text{if } r = n-s.
\end{cases}
\]
\end{corollary}

\begin{proof}
    The projection $\pi:X \rightarrow H/P_{s,r}$ gives a projective bundle with fibre $\mathbb P(M_{s-p, r-p})$. So from the proof of Lemma \ref{lem:linebundle}, we have  \[Pic(X) \simeq Pic(H/P_{s,r}) \oplus \mathbb Z.\] 
    Since $P_{s,r}$ is a maximal parabolic of $H$ when $r \neq n-s$, we have $Pic(H/P_{s,r})=\mathbb Z$.  

    So, $Pic(X) \simeq \mathbb Z \oplus \mathbb Z$ when $r \neq n-s$. When $r=n-s$, as $rs <n$, we have $(r,s)=(1,n-1)$ or $(r,s)=(n-1,1)$ and in these cases we have $Pic(X) \simeq \mathbb Z$.
\end{proof}

Recall that a smooth projective variety $Z$ is called \emph{Fano} 
if its anticanonical line bundle $\omega_Z^{-1}$ is ample.

\begin{theorem}\label{thm:fano}
    The variety $X$ is Fano.
\end{theorem}
\begin{proof} Since the projection $\pi:X \rightarrow H/P_{s,r}$ gives a projective bundle with fibre $\mathbb P(M_{s-p,r-p})$, the anticannonical line bundle on $X$ can be written as \[
\omega_X^{-1} \cong \pi^*(\omega_{H/P_{s,r}}^{-1}) \otimes \omega_{X/H/P_{s,r}}^{-1}, 
\] where $\omega_{X/H/P_{s,r}}$ is the relative canonical bundle. 
By Lemma \ref{lem:parabolic}, $P_{s,r}$ is a maximal parabolic subgroup of $H$ if $r \neq n-s$ and it is equal to the whole $H$ when $r=n-s$. So $H/P_{s,r}$ is a Grassmanian if $r \neq n-s$. Then by \cite[Chapter 3,
Section 3.1]{BK}, we have $\omega_{H/P}^{-1} \simeq \mathcal L_{2\rho_{P_{s,r}}}$ where, $\rho_{P_{s,r}}$ is the sum of the fundamental weights corresponding to the simple roots not in the positive roots $R_{P_{s,r}}^+$ of $P_{s,r}$. That is $\omega_{H/P_{s,r}}^{-1} \simeq \mathcal L_{2\varpi_{k}}$, which is ample as $2\varpi_{k}$ is dominant.

Since the fibre is $\mathbb P(M_{s-p, r-p})$, we have $$\omega_{X/H/P_{s,r}}^{-1} \simeq \mathcal O_{\mathbb P(M_{s-p, r-p})} ((r-p)(s-p)),$$ which is ample. Hence $\omega_X^{-1}$ is ample and hence $X$ is Fano. 

When $r=n-s$ we have $X \simeq \mathbb P^{rs-1}$, which is Fano. 
\end{proof}

\subsection{Cohomology of line bundles}

We now describe the cohomology of line bundles on the GIT quotient $X$.

\begin{theorem}\label{thm:cohomology} Let $\mathcal L$ be a line bundle on $X$. Let $a, b\in \mathbb Z$ as in Lemma \ref{lem:linebundle}. Then we have the following isomorphism as vector spaces:
    $$H^i(X, \mathcal L)= \bigoplus_{p+q=i} H^p(H/P_{s,r}, \mathcal L({b\varpi_{k}})) \otimes H^q(Y, \mathcal O_Y(a))$$ for all $i \geq 0$. 
\end{theorem}
\begin{proof}
By Theorem \ref{thm:structure}, we have 
\[
X \cong H \times_{P_{s,r}} Y.
\] 
Thus, $X$ is a Zariski locally trivial fibration with fibre $Y$.
If $n=r+s$, then $X\simeq \mathbb P^{rs-1}$, the result follows. 

Now assume that $n\neq r+s$. The base $H/P_{s,r}$ admits the Schubert cell decomposition
\[
H/P_{s,r} = \bigsqcup_{w \in W^{P_{s,r}}} C_w,
\] 
where each Schubert cell $C_w$ is a locally closed, affine, irreducible variety. Since these cells are smooth, they are local complete intersections inside suitable affine charts. Hence the base satisfies the hypotheses of \cite[Theorem B]{BD20}.

On the fibre side, the projective variety $\mathbb{P}(M_{s-p, r-p})$ admits the natural \emph{rank stratification}. For each integer $1 \le t \le \min(r-p,s-p)$, define
\[
Y_t = \{ [A] \in \mathbb{P}(M_{s-p, r-p}) \mid \operatorname{rk}(A) = t \}.
\] 
Each $Y_t$ is a locally closed irreducible subset, and its closure 
\[
Z_t = \{ [A] \in \mathbb{P}(M_{s-p, r-p}) \mid \operatorname{rk}(A) \le t \}
\] 
is the projectivized determinantal variety cut out, in affine charts, by the vanishing of all $(t+1)\times(t+1)$ minors. These determinantal loci are classical examples of local complete intersections in affine neighbourhoods. Thus, the rank strata $\{Y_t\}_{1 \leq t \leq min\{r-p,s-p\}}$ provide a stratification of the fibre into locally closed, affine, irreducible, local complete intersection subvarieties.
Therefore, the total space  
\[
X \;\cong\; H \times_{P_{s,r}} \mathbb{P}(M_{s-p, r-p})
\]
is stratified by the products $C_w \times Y_t$, where $C_w$ ranges over Schubert cells in $H/P_{s,r}$ and $Y_t$ ranges over coordinate strata in $\mathbb{P}(M_{s-p, r-p})$. In particular, there are open embeddings  
\[
U^1_w \times U^2_t \;\longrightarrow\; E,
\]
where $U^1_w \subseteq H/P_{s,r}$ is an affine chart containing the Schubert cell $C_w$ and $U^2_t \subseteq \mathbb{P}(M_{s-p, r-p})$ is a coordinate affine chart containing the stratum $Y_t$. Applying \cite[Theorem B]{BD20}, we obtain a convergent spectral sequence
\[
E_2^{p,q} = H^p\big(H/P_{s,r}, \mathcal{L}({b\varpi_{k}})\big) \otimes H^q\big(\mathbb{P}(M_{s-p, r-p}), \mathcal{O}_Y(a)\big)  
\;\;\Longrightarrow\;\; H^{p+q}\big(X, \mathcal{L}({b\varpi_{k}}) \otimes \mathcal{O}_Y(a)\big).
\]
Finally, by Borel-Weil-Bott theorem, the line bundle $\mathcal{L}(b\varpi_{k})$ on the base $H/P_{s,r}$ has non zero cohomology in atmost one degree. Thus, the spectral sequence collapses and the proof is complete.
\end{proof}

\begin{remark}
    Using the cohomology of line bundles on flag varieties, we can explicitly compute these cohomology groups for given $a, b \in \mathbb{Z}$.
\end{remark}

In paritular, we have 
\begin{corollary}
    Assume that $\mathcal L$ is numerically effective (nef) line bundle on $X$, then we have 
    \[H^i(X, \mathcal L)=0\quad \text{for all}\quad i>0.\] 
\end{corollary}
\begin{proof}
    If $n=r+s$, then $X\simeq \mathbb P^{rs-1}$, hence the result follows. Now assume that $n\neq r+s$, then the nef line bundle on $X$ is of the form \[\mathcal L= \mathcal{L}(b\varpi_{k}) \otimes \mathcal{O}_Y(a) \quad \text{with}\quad a, b\in \mathbb Z_{>0}.\]
      Now by Borel-Weil-Bott Theorem, the result follows. 
\end{proof}

\subsection{Automorphism group of $X$} 
In this subsection, we compute the connected component containing the identity of the automorphism group of the GIT quotient $$X:= \lambda \gitquot G(r,n)_{\lambda}^{ss}(\mathcal{L}(\varpi_r))
,$$ where we assume that either $p=0$ or $p=r+s-n$ .

\begin{theorem}\label{thm:aut}
$Aut^0(X) \simeq PSL(s, \mathbb C) \times PSL(n-s, \mathbb C)$.
\end{theorem}
\begin{proof}
By Theorem \ref{thm:structure}, we have $$X \simeq H \times_{P_{s,r}} \mathbb P(M_{s-p, r-p}),$$ where $H=SL(s, \mathbb C) \times SL(n-s, \mathbb C)$, $P_{s,r} = \text{Stab}_G(X(w_{s,r})) \cap H,$ and $M_{s-p, r-p}$ is the space of $s-p \times r-p$ matrices. Then the projection $\pi:X \rightarrow H/P_{s,r}$ is a fibre bundle with fibre $\mathbb P(M_{s-p, r-p})$.
Since $\pi_*\mathcal O_X=\mathcal O_{H/P_{s,r}}$, by \cite[Corollary 2.2]{Br}, any automorphism $f \in Aut^0(X)$
induces an automorphism $g \in Aut^0(H/P_{s,r})$ of the base $H/P_{s,r}$ satisfying:
\[
\pi \circ f = g \circ \pi.
\]
Hence, we have a natural algebraic group homomorphism
\[
\pi_*:Aut^0(X) \;\longrightarrow\; Aut^0(H/P_{s,r}).
\]
The group $H$ acts on $X = H \times_{P_{s,r}} Y$ by left multiplication on the first factor. 
This action projects to the natural action of $H$ on the base $H/P_{s,r}$, which is transitive. Since $P_{s,r}$ is a parabolic subgroup of $H$, $H/P_{s,r}$ is a partial flag variety. Since the
center $Z(H)$ of $H$ act trivially on $H/P_{s,r}$,
by a theorem of Demazure (see \cite{Dem}), the image of $H/Z(H)$ in $Aut(H/P_{s,r})$ is exactly $Aut^0(H/P_{s,r})$. 
Hence, every element of $Aut^0(H/P_{s,r})$ comes from some element of $H/Z(H)$, 
and this element also defines an automorphism of $X$. Therefore, every element of $Aut^0(H/P_{s,r})$ has a lift in $Aut^0(X)$ and so $\pi_*$ is surjective.

The kernel of $\pi_*$ is the group of relative automorphisms
\[
Aut^0_{H/P_{s,r}}(X) := \{ f \in Aut(X) \mid \pi \circ f = \pi \},
\]
i.e.\ automorphisms of $X$ that preserve every fibre of $\pi$. In other words this is the set of all $P_{s,r}$-equivariant automorphisms of $\mathbb P(M_{s-p, r-p})$. Thus, we obtain a short exact sequence
\[
1 \;\longrightarrow\; Aut^0_{H/P_{s,r}}(X) \;\longrightarrow\; Aut^0(X) \;\longrightarrow\; H/Z(H) \longrightarrow\; 1.
\]
Now, since the action of $P_{s,r}$ on $M_{s-p, r-p}$ is irredeucible, by Schur's lemma we have \[End_{P_{s,r}}(M_{s-p, r-p})=\mathbb C\] and so $Aut^0_{H/P_{s,r}}(X)=PGL(M_{s-p, r-p})^{P_{s,r}}=\{1\}$. Thus, we conclude that \[Aut^0(X)=H/Z(H)=PSL(s, \mathbb C) \times PSL(n-s, \mathbb C).  \]
\end{proof}

\subsection{{The $H$-module structure}}
In this subsection, we study the $H$-module structure of $H^0(X, \mathcal{M}^{\otimes m})$, where $\mathcal{M}$ denotes the descent of the line bundle $\mathcal{L}(n\varpi_r)$ from $Gr(r, n)$ to $X$. Note that $G(r,n)_{\lambda}^{ss}(\mathcal{L}(\varpi_r)) = G(r,n)_{\lambda}^{ss}(\mathcal{L}(n\varpi_{r})).$

By Proposition~\ref{prop:sphericality}, the quotient $X$ is a spherical $H$-variety. Therefore,  the $H$-module $H^0(X, \mathcal{M}^{\otimes m})$ is multiplicity-free for all $m$ (see \cite[Theorem 2.1.2]{Perrin}).
Note that
$$
H^0(X, \mathcal{M}^{\otimes m}) = H^0\big(Gr(r,n), \mathcal{L}(n\varpi_r)^{\otimes m}\big)^{\lambda}.
$$
Let $K$ be a semi-simple algebraic group over $\mathbb C$ with a maximal torus $T'$ and a Borel subgroup $B'$ containing $T'$. Let $S = \{\beta_1,\beta_2,\ldots,\beta_n\}$ be the set of simple roots. 
Let $Q$ be a maximal parabolic subgroup of $K$ corresponding to a simple roots say $\beta_r$. Let $L$ be the Levi subgroup of $Q$ containing the maximal torus $T'$. Let $B_L = B' \cap L$ be the Borel subgroup of $L$.

We start with the following general result:
\begin{lemma}\label{red}
 Let $\mu$ be a dominant character of $T'$ such that $\langle \mu, \check{\beta_r}\rangle = 0$, and let $V_L(\mu)$ be the irreducible $L$-module of highest weight $\mu$. Then for any $b \in \mathbb{Z}_{\geq 0}$, we have:
$$
H^0(K/Q, \mathcal{L}(V_L(\mu) \otimes \mathbb{C}_{b\varpi_r'})) = H^0(K/B', \mathcal{L}(\mu + b\varpi_r')).
$$
\end{lemma}

\begin{proof}
For $b \in \mathbb{Z}_{\geq 0}$, since $\langle b\varpi_r', \check{\beta} \rangle = 0$ for all $\beta \neq \beta_r$, it follows that
$$
H^0(Q/B', \mathcal{L}(\mu + b\varpi_r')) = H^0(L/B_L, \mathcal{L}(\mu + b\varpi_r')) = H^0(L/B_L, \mathcal{L}(\mu)) \otimes \mathbb{C}_{b\varpi_r'}.
$$
Hence,
$$
H^0(K/B', \mathcal{L}(\mu + b\varpi_r')) = H^0\left(K/Q, \mathcal{L}\left(H^0(Q/B', \mathcal{L}(\mu + b\varpi_r'))\right)\right),$$ where $\mathcal{L}\left(H^0(Q/B', \mathcal{L}(\mu + b\varpi_r'))\right)$ is the homogeneous line bundle on $K/Q$ associated with the $Q$-module $H^0(Q/B', \mathcal{L}(\mu + b\varpi_r'))$.
Thus, by the above identity, we obtain:
\begin{align*}
H^0(K/B', \mathcal{L}(\mu + b\varpi_r')) 
&= H^0\left(K/Q, \mathcal{L}\left(H^0(L/B_L, \mathcal{L}(\mu + b\varpi_r'))\right)\right) \\
&= H^0(K/Q, \mathcal{L}\left(H^0(L/B_L, \mathcal{L}(\mu)) \otimes \mathbb{C}_{b\varpi_r'}\right) \\
&= H^0(K/Q, \mathcal{L}(V_L(\mu)^*) \otimes \mathbb{C}_{b\varpi_r'}),
\end{align*}
as claimed.
\end{proof}

Let $\mathcal M$ be the descent of $\mathcal{L}(n\varpi_r)$ to $X$. 
We now give a decomposition of the $H$-module $H^0(X, \mathcal{M}^{\otimes m})$ for any $m\in \mathbb{Z}_{\geq 0}$. Let $T_2$ be the maximal torus of $SL(n-s, \mathbb C)$ contained in $T$.

\begin{theorem}\label{thm:G decomposition}
We keep the notation as above. Then, we have
$$
H^0(X, \mathcal{M}) = \bigoplus_{\nu \leq a\varpi_1} H^0(H/(B \cap H), \mathcal{L}(\nu, \overline{\nu} + b\varpi_{k})),
$$
where $\overline{\nu}$ is the dominant character of $T_2$ determined by $\nu$.
\end{theorem}

\begin{proof}

Note that by Proposition \ref{Prop:dual}, we have $$\lambda \gitquot G(r,n)_{\lambda}^{ss}(\mathcal{L}(\varpi_r)) \simeq \lambda' \gitquot G(n-r,n)_{\lambda'}^{ss}(\mathcal{L}(\varpi_{n-r})), $$ where $\lambda'=n\check{\varpi}_{n-s}$. 
Thus, we can assume that $r+s \leq n$.
   By Theorem \ref{thm:structure}, we have $X = H \times_{P_{s,r}} Y$, where $P_{s,r} = \text{Stab}_G(X(w_{r,s})) \cap H$ and $Y=\mathbb P(M_{s-p, r-p})$. Let  $\mathcal{M}_Y$ be the descent of $\mathcal{L}(n\varpi_r)$ to $Y$. Then by \cite[Theorem 3.3 and Corollary 3.4]{DEP80}, the $SL(s-p, \mathbb C)\times SL(r-p, \mathbb C)$-module of global sections of $Y$ is
\[
H^0\bigl(Y,\mathcal M_Y\bigr)
\;\cong\;
\bigoplus_{\nu\le a\varpi_1} \; V_{SL(s-p)}(\nu)^* \otimes V_{SL(r-p)}(\overline\nu),
\]
where the sum runs over dominant weights $\nu$ for $SL(s-p, \mathbb C)$ with $length(\nu) \leq min\{s-p,r-p\}$ and satisfying the
bound $\nu\le a\varpi_1$ (equivalently, partitions with at most $a$
columns), and $\overline\nu$ denotes the naturally corresponding dominant
weight of $SL(r-p, \mathbb C)$.

Consider the associated vector bundle on the base $H/P_{s,r}$ with fibre
$H^0(Y,\mathcal M_Y)$ and twisted by the character $\mathbb C_{b\varpi_k}$ arising
from the linearization.  We then have
\[
H^0\bigl(X,\mathcal M\bigr)
\;\cong\;
H^0\Bigl(H/P_{s,r},\,H\times_{P_{s,r}}\bigl(H^0(Y,\mathcal M_Y)\otimes\mathbb C_{b\varpi_k}\bigr)\Bigr).
\]
Inserting the decomposition of $H^0(Y,\mathcal M_Y)$, we obtain
\begin{align*}
H^0(X,\mathcal M)
&\;\cong\;
\bigoplus_{\nu\le a\varpi_1}
H^0\Bigl(H/P_{s,r},\,H\times_{P_{s,r}}\bigl(V(\nu)^*\otimes V(\overline\nu)\otimes \mathbb C_{b\varpi_k}\bigr)\Bigr) \\
&\;\cong\; \bigoplus_{\nu \leq a\varpi_1} H^0\left(H/P_{s,r}, \mathcal{L}(V(\nu)^*\otimes V(\overline\nu)\otimes \mathbb C_{b\varpi_k})\right).
\end{align*}
By using Lemma \ref{red}, we get that 
\[H^0(X, \mathcal{M}) 
\cong \bigoplus_{\nu \leq a\varpi_1} H^0\left(H/(B \cap H), \mathcal{L}(\nu, \overline{\nu} + b\varpi_{k})\right).\]
This proves the theorem.
\end{proof}

\subsection{Projective Normality}
In this subsection, we prove the projective normality of $X$ with respect to a suitable ample line bundle $\mathcal{M}$. 

We begin by recalling the definition of projective normality for a projective variety. A projective variety $X \subset \mathbb{P}^n$ is said to be \emph{projectively normal} if the affine cone $\hat{X}$ over $X$ is normal at its vertex. For a reference, see \cite[Ex.3.18, p.23]{RH}.

For practical purposes, we use the following criterion for projective normality of a polarized variety. A polarized variety $(X, \mathcal{L})$, where $\mathcal{L}$ is a very ample line bundle, is said to be projectively normal if its homogeneous coordinate ring 
\[
\bigoplus_{m \in \mathbb{Z}_{\geq 0}} H^0(X, \mathcal{L}^{\otimes m})
\]
is integrally closed, and it is generated as a $\mathbb{C}$-algebra by $H^0(X, \mathcal{L})$ (see \cite[Ex.5.14, Chap.II]{RH}). We note that projective normality depends on the specific projective embedding of the variety.

\begin{theorem}\label{thm:projective_normality}
    Let $\mathcal{M}$ be the descent of $\mathcal{L}({n\varpi_r})$ to $X$. Then, the polarized variety $(X, \mathcal{M})$ is projectively normal.
\end{theorem}

\begin{proof}
    Since the Grassmannian $G(r,n)$ is normal, it follows that the GIT quotient $X$ is also normal. Consider the homogeneous coordinate ring of $(X, \mathcal{M})$:
    $$
    \bigoplus_{d \in \mathbb{Z}_{\geq 0}} H^0(X, \mathcal{M}^{\otimes d}).
    $$
        Write 
    $$
    R := \bigoplus_{d \in \mathbb{Z}_{\geq 0}} R_d, \quad \text{where} \quad R_d := H^0(X, \mathcal{M}^{\otimes d}).
    $$
    To prove the projective normality of $(X, \mathcal{M})$, it is enough to show that $R$ is generated by $R_1$ as a $\mathbb{C}$-algebra. To this end, we first prove that the multiplication map
    \begin{equation}\label{Eq:projnormal}
      \psi: H^0(X, \mathcal{M})^{\otimes m}\rightarrow H^0(X, \mathcal{M}^{\otimes m})
    \end{equation}
    is surjective for all $m \in \mathbb Z_{\geq 0}$.
    Recall from Theorem \ref{thm:structure} that we have 
    $X \cong H \times_{P_{s,r}} Y$, where $Y = \mathbb{P}(M_{s-p, r-p})$. Moreover, by Lemma \ref{lem:linebundle}, we can write
    $$
    \mathcal{M} = \mathcal{L}(b\varpi_{k}) \otimes \mathcal{O}_Y(a), \quad \text{for some } a, b \in \mathbb{Z}.
    $$
    Further, since $\mathcal M$ is ample we have $a,b \in \mathbb N$. Therefore, by Theorem \ref{thm:cohomology}, we have
    \begin{equation}\label{Eq:coho}
        H^0(X, \mathcal{M}) = H^0(H/P_{s,r}, \mathcal{L}({b\varpi_{k}})) \otimes H^0(Y, \mathcal{O}_Y(a)).
    \end{equation}
    Since the polarized varieties $(H/P_{s,r}, \mathcal{L}({b\varpi_{k}}))$ and $(Y, \mathcal{O}_Y(a))$ are projectively normal, the multiplication maps
    $$
    H^0(H/P_{s,r}, \mathcal{L}({b\varpi_{k}}))^{\otimes m} \rightarrow H^0(H/P_{s,r}, \mathcal{L}({b\varpi_{k}})^{\otimes m})
    $$
    and
    $$
    H^0(Y, \mathcal{O}_Y(a))^{\otimes m} \rightarrow H^0(Y, \mathcal{O}_Y(ma))
    $$
    are surjective. Thus, by Equation \eqref{Eq:coho}, we get the map $\psi$
       in Equation \eqref{Eq:projnormal} is surjective.
    Now, by induction, we conclude that $R$ is generated by $R_1$. Therefore, using \cite[Ex.~5.14, Chap.~II]{RH}, we conclude that the polarized variety $(X, \mathcal{M})$ is projectively normal.
\end{proof}




\end{document}